\documentclass[11pt,a4paper]{amsart}
\usepackage{amsmath}
\usepackage{amsfonts}
\usepackage{comment}
\usepackage{amssymb}
\usepackage{enumerate}

\usepackage[left=2.5cm,right=2.5cm,top=2.5cm,bottom=2.5cm]{geometry}
\usepackage{amsthm}
\usepackage{tikz}
\usetikzlibrary{decorations.pathreplacing,calligraphy,shapes.misc}

\newcommand{\C}{\mathbb{C}}

\newcommand{\R}{\mathbb{R}}

\newcommand{\ra}{\rightarrow}

\newcommand{\SU}{\textup{SU}}

 \DeclareMathOperator{\Reom}{Re\,\Omega}
\DeclareMathOperator{\Imom}{Im\,\Omega}
\DeclareMathOperator{\Bsev}{\mathbb{B}_7}
\DeclareMathOperator{\Csev}{\mathbb{C}_7}
\DeclareMathOperator{\Dsev}{\mathbb{D}_7}
\usepackage{graphicx}
\usepackage{tikz-cd} 
\theoremstyle{definition}
\newtheorem{definition}{Definition}[section]

\theoremstyle{plain}
\newtheorem{theorem}{Theorem}[section]

\newtheorem{lemma}[theorem]{Lemma}

\newtheorem{prop}[theorem]{Proposition}

\newtheorem{thmx}{Theorem}

\newtheorem{remark}[theorem]{Remark}
\theoremstyle{plain}

\usepackage[foot]{amsaddr}
\author{Jakob Stein}
\address{University College London, University of Bath}
\email{jakob.stein.19@ucl.ac.uk}
\author{Matt Turner}
\email{mt916@bath.ac.uk}
\date{}
\title{$G_2$-INSTANTONS ON THE SPINOR BUNDLE OF THE 3-SPHERE.}
\begin{document}
\maketitle
\begin{abstract}
We classify $G_2$-instantons admitting $\SU(2)^3$-symmetries, and construct a new family of examples on the spinor bundle of the 3-sphere, equipped with the asymptotically conical, co-homogeneity one $G_2$-metric of Bryant-Salamon. We also show that outside of the $\SU(2)^3$-invariant examples, any other $G_2$-instanton on this metric with the same asymptotic behaviour must have obstructed deformations.
\end{abstract}

\section{Introduction}
Let $\left(M^7,\varphi \right)$ be a $G_2$-manifold, equipped with a principal $G$-bundle $P \rightarrow M$ for a compact, semi-simple Lie group $G$, where $\varphi \in \Omega^3 (M)$ is a torsion-free $G_2$-structure on $M$. A connection $A$ on $P$ is called a $G_2$-\textit{instanton} if it satisfies the $G_2$-\textit{instanton equations}, namely 
\begin{align} \label{eq:g2instantons00}
F_A \wedge * \varphi = 0 
\end{align}
where $*$ is the Hodge star of the Riemannian metric defined by $\varphi$, and $F_A \in \Omega^2 \left(M;\mathrm{ad} P \right)$ is the curvature of $A$.

These equations first appeared in \cite{corrigan:gaugefeilds,ward:instantons}, and generalise  \textit{anti-self-dual} (ASD) instantons found in dimension $4$: solutions to a first-order system of partial differential equations which minimise the Yang-Mills energy functional. In the conjectural picture outlined by Donaldson-Thomas in \cite{donaldson1996gauge} and later expanded upon in \cite{donaldson2009gauge} and \cite{walpuski:g2instantons}, the moduli-space of solutions to \eqref{eq:g2instantons00} could potentially be used to construct invariants for $G_2$-manifolds, analogous to the anti-self-dual invariants constructed in dimension $4$. However, due to the analytic difficulties involved, as explained in \cite{tian2000gauge}, there is a need for a more complete understanding of the behaviour of these solutions.

In this note, we will exploit symmetries of both the bundle data and the underlying Riemannian manifold, in order to construct new examples of solutions of \eqref{eq:g2instantons00} on the spinor bundle of the 3-sphere, equipped with the metric of Bryant-Salamon \cite{bryantsalamon}. Restricting to this setting and structure group $\SU(2)$, we will be able to give explicit descriptions of the moduli space of these symmetric solutions.  

\subsection*{Families of $G_2$-metrics}

Since metrics with holonomy contained in $G_2$ are Ricci-flat, then besides the flat metrics, the maximal symmetries we could hope to exploit for a non-compact manifold are \textit{co-homogeneity one}, i.e. there is a Lie group of isometries acting on the Riemannian manifold with generic orbits of co-dimension one. There is now an infinite collection of one-parameter families of complete co-homogeneity one $G_2$-metrics, recently constructed by Foscolo-Haskins-Nordstr\"om in \cite{Foscolo2018}. Each one-parameter family is parameterised by some $\alpha$ in the interval $(0,1]$ and a generic member of each family has asympotically locally conical (ALC) geometry, i.e. it is asymptotic to a metric on a circle bundle over a 6-dimensional Calabi-Yau cone. At either end of the parameter space, the geometry transitions: as $\alpha\to0$, the $G_2$-structure collapses to a Calabi-Yau structure on the 6-dimensional cone, while at $\alpha=1$, the metric is asymptotically conical (AC). 

The only family of this form to predate \cite{Foscolo2018} is referred to as the $\Bsev$ family in the physics literature  \cite{gibbonspope:g2mtheory}. It was predicted to exist in 2001 by Brandhuber–Gomis–Gubser–Gukov \cite{bggg} and constructed in 2013 by Bogoyavlenskaya \cite{bogoyavg2}. Foscolo-Haskins-Nordstr\"om recover the $\Bsev$ family, construct the $\Csev$ and $\Dsev$ families predicted to exist in \cite{gibbonspope:g2mtheory}, and construct infinitely many more families that can be viewed as variations of the $\Csev$ family.

The families constructed in \cite{Foscolo2018} can be viewed as desingularisations of conically singular ALC metrics, by removing a neighbourhood of the singularity and gluing in a rescaled AC manifold. This method adapts earlier arguments of Karigiannis in \cite{karigiannis:desingularizeg2}, and the choice of desingularisation yields different families of ALC metrics. However, three of these families, $\Bsev$  and two variants of $\Dsev$, share the same limiting complete AC $G_2$ metric: the metric on the spinor bundle $\mathbf{S}(S^3)$ constructed by Bryant-Salamon in \cite{bryantsalamon}, which has a co-homogeneity one action of $\SU(2)^3$. In the former case, as $\alpha\to0$, the metric collapses to the Stenzel metric on the smoothing of the conifold, while each of the $\Dsev$ families collapse to a different small resolution of the conifold.  

\tikzset{
  mynode/.style={fill,circle,inner sep=2pt,outer sep=0pt}
}

\begin{figure}[h]
\centering
\begin{tikzpicture}
\draw[black,thick,->] (0,0) -- (10,0)
    node[pos=0,mynode,fill=black,label=above:\textcolor{violet}{Collapsed to Stenzel}]{}
    node[pos=0.6,mynode,fill=black,text=blue,label=above:\textcolor{purple}{AC}, label=below:{BS metric}]{}
    node[pos=0.3,label=above:\textcolor{teal}{ALC}]{}
    node[pos=1,label=below:\textcolor{black}{$\alpha$}]{}
    node[pos=0.8,label=above:\textcolor{black}{Incomplete}]{};
\end{tikzpicture}
\caption{Bryant-Salamon manifold in the $\Bsev$ family}\label{fig:BS}
\end{figure}
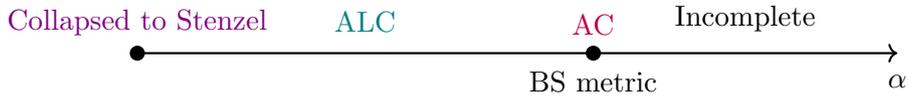

The Bryant-Salamon metric on $\mathbf{S}(S^3)$ is asymptotic to the cone over the 6-manifold $N=S^3\times S^3$. $N$ has symmetry group $\SU(2)^3\times S_3$ where $S_3$ is the group of permutations on 3 elements. The three variants of the Bryant-Salamon metric at the limit of the $\Bsev$ and $\Dsev$ families are diffeomorphic but not equivariantly diffeomorphic with respect to the cohomogeneity one group action. The even elements of $S_3$ yield these three realisations of the metric while the transpositions are orientation reversing isometries on each metric. This symmetrical picture is described in more detail by Atiyah-Witten in \cite{AW01}, and will be apparent in \S \ref{section:g2instantonodes}.

\subsection*{$G_2$ gauge theory}

We consider $G_2$-instantons on $\mathbf{S} (S^3)$ with its AC Bryant-Salamon $G_2$-metric to answer a question posed in \cite{goncalog2}, by constructing a new one-parameter family of $G_2$-instantons and classifying $\SU(2)^3$-invariant $G_2$-instantons satisfying a natural curvature decay condition. Moreover, in \S \ref{section:deform}, using the deformation theory of instantons on AC $G_2$-metrics developed in \cite{driscollthesis}, we show that the symmetric solutions from \S \ref{section:g2instantonodes} are the only solutions with unobstructed deformations in the moduli-space of instantons sharing the same asymptotic behaviour.  

The first $G_2$-instantons found on the Bryant-Salamon manifold were constructed in 2014 by Clarke in \cite{clarke}. They are parameterised by the interval $[0,\infty)$ and exist on one of the two possible $\SU(2)^3$-invariant $\SU(2)$-bundles over $\mathbf{S}(S^3)$. Lotay-Oliveira found in \cite{goncalog2} a single solution on the other $\SU(2)$-bundle and showed that, outside of a compact subset, this solution is the limit of the family constructed by Clarke. In this paper, we show that this single solution actually lies at the centre of a 1-parameter family of solutions, parameterised by the interval $[-1,1]$. The resulting moduli space of invariant $G_2$-instantons is shown in Figure \ref{fig:mod}.

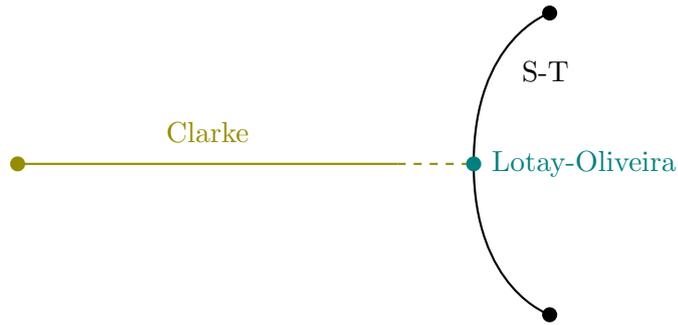
\begin{figure}[h]
\centering
\begin{tikzpicture}
\draw[olive,thick,-] (0,0) -- (5,0)
    node[pos=0,mynode,fill=olive]{}
    node[pos=0.5,label=above:\textcolor{olive}{Clarke}]{};

\draw [black,thick,tension=1.5] plot [smooth] coordinates {(7,-2) (6,0) (7,2)}
node[mynode,fill=black]{};
\draw (7,-2) node[mynode] {};
\draw[olive,thick,dashed] (5,0) -- (6,0)
node[pos=1,mynode,fill=teal,label=right:\textcolor{teal}{Lotay-Oliveira}]{};
\node[below right] at (6.5,1.5) {S-T};
\end{tikzpicture}
\caption{The moduli space of invariant $G_2$-instantons on $\mathbf{S}(S^3)$ with the Bryant-Salamon metric.}\label{fig:mod}
\end{figure}

Constructing $G_2$-instantons on the ALC members of these families is a harder problem. The only known examples of $G_2$-instantons on these ALC $G_2$-manifolds were found by Lotay-Oliveira in \cite{goncalog2}. They exist on the only ALC metric of the $\Bsev$ family that is known explicitly, namely the BGGG metric constructed by Brandhuber–Gomis–Gubser–Gukov in \cite{bggg}.  

With a complete understanding of the moduli space of invariant $G_2$-instantons on the Bryant-Salamon $\mathbf{S}(S^3)$, we may be able to construct instantons on the ALC metrics of the $\mathbb{B}_7$ and $\mathbb{D}_7$ families sufficiently close to the AC limit\footnote{See \cite{stein:calabiyaugauge} for a discussion of constructing examples near the collapsed limit.} via a gluing procedure. Recall that these families may be viewed as desingularisations of the incomplete conically-singular (CS) co-homogeneity one $G_2$-metrics constructed in \cite{Foscolo2018}. By considering instantons on these CS metrics and the instantons on the Bryant-Salamon AC metric, we may be able to follow a similar gluing procedure to produce instantons on ALC members of these families.  

\subsection*{Main results and plan of the paper}

We will first give some preliminary details in \S \ref{section:G2define} about $G_2$-structures and, more specifically, the Bryant-Salamon metric in \S \ref{section:bryantsalamon}. We begin \S \ref{section:G2overview} with an overview of the $G_2$-instanton equations, before focusing on the invariant setting in \S \ref{section:g2instantonodes}. We will recall the system of ODEs corresponding to the $\SU(2)^3$-invariant equations from \cite{goncalog2} in Proposition \ref{prop:g2instode}, and the parametrization of its local solutions in Proposition \ref{prop:G2localsol}. 

The new results contained in this section are outlined in Theorem \ref{thm:g2instintro}, classifying global solutions to the $G_2$-instanton ODEs using the theory of asymptotically autonomous ODE systems from  \cite{markusdiffsystem}.

\begin{thmx} \label{thm:g2instintro} $\SU(2)^3$-invariant instantons on $\mathbf{S}(S^3)$ with gauge group $\SU(2)$, and quadratic curvature decay are in two one-parameter families $T_\gamma$, $\gamma \in \left[ 0, \infty \right)$, and $T'_{\gamma'}$, $\gamma' \in \left[-1, 1 \right]$. Moreover
\vspace{-0.3cm}
\begin{enumerate}
\item The isometry exchanging the factors of $\SU(2)^2$ on the principal orbits sends $T'_{\gamma'} \mapsto T'_{-\gamma'}$;  
\item $T_0$, $T'_1$, $T'_{-1}$ are flat, otherwise $T_\gamma$, $T'_{\gamma'}$ are irreducible;
\item The irreducible $T_\gamma$, $T'_{\gamma'}$ are asymptotic to the pull-back of the unique nearly K\"ahler instanton on $S^3\times S^3$ with rate $-3$.     
\end{enumerate}
\end{thmx}

The family $T_{\gamma}$ was constructed by Clarke in \cite{clarke} and it was shown in \cite{goncalog2} that this family bubbles off an ASD instanton transverse to the associative submanifold $S^3$. In this case, the curvature concentrates on the singular orbit. In contrast, the curvature of the instantons in the family $T'_{\gamma'}$ concentrates further along the asymptotic end as $\gamma'\to\pm1$. Indeed, we will see that we can view instantons close to these limits as a gluing of the pull-back of an instanton on the cone to a perturbation of the flat connections $T'_{\pm1}$. A similar gluing procedure can be seen in the construction of $G_2$-instantons in \cite{turner:c7familyinstantons} on the $\C_7$ family of $G_2$-metrics.

In the final section \S \ref{section:deform}, we use a computation of \cite{driscollthesis} to show Proposition \ref{prop:deform}, which describes the moduli-space of $G_2$-instantons on $\mathbf{S} (S^3)$ away from the invariant regime of \S \ref{section:g2instantonodes}. The proposition states that un-obstructed $G_2$-instantons on $\mathbf{S} (S^3)$ asymptotic to the nearly K\"ahler instanton on $S^3\times S^3$ with rate $\mu>-3$ are $\SU(2)^3$-invariant, and hence classified by Theorem \ref{thm:g2instintro}.  

\subsection*{Acknowledgements} Special thanks to Simon Salamon, Gon\c calo Oliveira, Lorenzo Foscolo, Jason Lotay and Johannes Nordstr\"om for their helpful comments and discussions. The first author was funded by the Royal Society, through a studentship supported by the Research Fellows Enhancement Award 2017 RGF-EA-180171, and the EPSRC through the UCL Research Associates Award EP-W522636-1. The second author was funded by the EPSRC Studentship 2106787 and the Simons Collaboration on Special
Holonomy in Geometry, Analysis and Physics \#488631.

\section{$G_2$ structures}\label{section:G2define}

Recall that a $G_2$-structure on a 7-manifold $M$ is a reduction of the frame bundle to the exceptional Lie group $G_2$. This is equivalent to the existence of a non-degenerate $3$-form $\varphi \in \Omega^3 (M)$ which is fixed by the point-wise action of $G_2 \subset GL(7,\R)$ in some framing of the tangent space at each point. Furthermore, this data determines a Riemannian metric, volume-form, and orientation on $M$.

 Existence of such a structure on some oriented 7-manifold is purely topological: it is equivalent to the existence of a spin structure \cite{salamon:riemanniangeometry}. On the other hand, constructing torsion-free $G_2$-structures can be much more difficult. 
\begin{definition} \label{def:g2}
A $G_2$-manifold $\left( M, \varphi\right)$ is a 7-manifold equipped with a torsion-free $G_2$-structure, i.e. $d \varphi = d^* \varphi = 0$. 
\end{definition}  

Although their existence was first suggested by the work of Berger \cite{berger:list}, the first examples of complete, irreducible $G_2$-manifolds were only constructed much later in \cite{bryantsalamon}, by exploiting co-homogeneity one symmetries. Subsequently, another co-homogeneity one example was found in \cite{bggg}, which was later generalised to a one-parameter family in \cite{bogoyavg2}, and a partial proof of existence for a second one-parameter family was given in \cite{bazaikin:c7family}. More recently, infinitely many families of complete, co-homogeneity one $G_2$-metrics have been found in \cite{Foscolo2018}, confirming earlier predictions in the physics literature \cite{gibbonspope:g2mtheory}. In order to understand these constructions, it will be useful to recall some definitions regarding $G_2$ geometry in co-dimension one.

Moreover, suppose $M$ admits a cohomogeneity one action of a Lie group $G$. Let $K_0$ be the stabiliser of the principal orbits of the action, and $K$ be the stabiliser of the singular orbit. In this note, we will encode the cohomogeneity one action in a group diagram, namely 
\[K_0\subset K\subset G.\]

\subsection{$G_2$-structure evolution equations}

If $\iota:N \hookrightarrow M$ is an oriented immersion of a $6$-manifold $N$ into $M$, then a $G_2$-structure $\varphi$ on $M$ naturally equips $N$ with an $\SU(3)$-structure, namely
\begin{align} \label{su3define}
\omega = \iota^* \left( \hat{n} \lrcorner\ \varphi \right)& &\Reom = \iota^* \varphi& &\Imom = \iota^* \left(- \hat{n} \lrcorner *\varphi \right)
\end{align}
where $\hat{n}$ the canonical unit normal to $\iota(N) \subset M$, defined by the chosen orientation and the Riemannian metric induced by $\varphi$.

If $\iota$ is an embedding, then we can view $N \subset M$ as an oriented hyper-surface, and a tubular  neighbourhood of $N\subset M$ can be identified with $N \times I$ for some interval $I \subseteq \R$ using the exponential map. In these coordinates, the metric on $M$ appears as $g = dt^2 + g_t$ for some $t$-dependent metric $g_t$ on $N$, and \eqref{su3define} gives rise to a family of $\SU(3)$-structures $\left(\omega, \Omega \right)_{t \in I}$ inducing $g_t$. Meanwhile, the $G_2$-structure on $M$ appears as 
\begin{align} \label{G2define}
\varphi = dt \wedge \omega + \Reom& &*\varphi = - dt \wedge \Imom + \tfrac{1}{2} \omega^2
\end{align}
and if this $G_2$-structure is torsion-free, then $\left( \omega, \Omega \right)_{t \in I}$ satisfy the following \textit{half-flat} structure equations on $N$ 
\begin{align} \label{halfflat}
d \omega \wedge \omega = 0& &d \Reom=0
\end{align}
subject to the evolution equations
\begin{align} \label{halfflatdynamic}
d \omega = \partial_t \Reom& &d \Imom = - \tfrac{1}{2} \partial_t \left( \omega^2 \right)
\end{align}
Observe that \eqref{halfflat} is preserved under \eqref{halfflatdynamic}, which allows us to interpret a torsion-free $G_2$-structure (at least locally) as a flow by \eqref{halfflat} in the space $\SU(3)$-structures satisfying \eqref{halfflat} on some fixed $6$-manifold, cf. \cite{hitchin:stableforms}.

We now consider a special case of equations satisfying both \eqref{halfflat} and \eqref{halfflatdynamic}. Let $N$ be a 6-manifold equipped with a fixed $\SU(3)$-structure $\left( \omega^{nK}, \Omega^{nK}\right)$ satisfying the structure equations 
\begin{align} \label{nearlykahler}
d \omega^{nK} = 3 \Reom^{nK}& &d\Imom^{nK} = -2 \left(\omega^{nK}\right)^2.
\end{align}
Then the 1-parameter family $\left(\omega, \Omega\right)_{t \in \mathbb{R}_{>0}}$ of $\SU(3)$-structures 
\begin{align} \label{nearlykahler2}
&\omega = t^2 \omega^{nK}  &\Omega = t^3 \Omega^{nK}
\end{align}
satisfy \eqref{halfflat} and \eqref{halfflatdynamic} if and only if $\left( \omega^{nK}, \Omega^{nK}\right)$ satisfies \eqref{nearlykahler}. 
As in \eqref{G2define}, this family defines the \textit{conical $G_2$-structure} $\varphi_C$ on $\mathbb{R}_{>0} \times N$, given by 
\begin{align} \label{G2cone}
\varphi_C = t^2 dt \wedge \omega^{nK} + t^3 \Reom^{nK}& &*\varphi_C = - t^3 dt \wedge \Imom^{nK} + \tfrac{1}{2} t^4 \left(\omega^{nK}\right)^2.
   \end{align}
It is torsion-free if and only if $\left( \omega^{nK}, \Omega^{nK}\right)$ satisfies the structure equations \eqref{nearlykahler}. We refer to an $\SU(3)$-structure $\left( \omega^{nK}, \Omega^{nK}\right)$ satisfying \eqref{nearlykahler} as being \textit{nearly-K\"{a}hler}; one can show that such an $\SU(3)$-structure induces a nearly-K\"{a}hler metric $g^{nK}$ on $N$ or, in other words, the $G_2$-metric $g_C$ induced by $\varphi_C$ on  $\mathbb{R}_{>0} \times N$ is a metric cone $g_C = dt^2 + t^2 g^{nK}$. 

\subsection{Bryant-Salamon $G_2$-manifold} \label{section:bryantsalamon}
The spinor bundle $\mathbf{S}(S^3)$ of the 3-sphere admits a one-parameter family of co-homogeneity one $G_2$-metrics described by Bryant-Salamon in \cite{bryantsalamon}. This parameter represents the volume of the zero-section $S^3 \subset \mathbf{S}(S^3)$, or alternatively, the coefficient of the cohomology class $\left[ \varphi \right]$ of $\mathbf{S}(S^3)$, and can be fixed up to diffeomorphisms by scaling the resulting metric. In this section, we will give a short exposition of this construction, following \cite{bryantsalamon}, \cite{goncalog2}, \cite{Foscolo2018}. 

The total space of the spinor bundle can be written homogeneously as \[\mathbf{S}(S^3) = \SU(2)^2\times_{\Delta \SU(2)}\mathbb{H},\] where $\Delta \SU(2)$ acts on the right diagonally, and it admits a co-homogeneity one action of $\SU(2)^3$, viewed here as acting on the left \cite[\S 3]{bryantsalamon}. The corresponding group diagram is 
\begin{align*}
\Delta_{1,2,3} \SU(2) \subset \Delta_{1,2} \SU(2) \times \SU(2) \subset \SU(2)^3
\end{align*}
where $\Delta_{1,2} \SU(2)$ and $\Delta_{1,2,3} \SU(2)$ denote the diagonal $\SU(2)$-subgroup in the first two factors of $\SU(2)^3$ and the diagonal subgroup in all three factors respectively.  

As a smooth manifold, $\mathbf{S}(S^3)$ is diffeomorphic to $S^3 \times \mathbb{R}^4$. This diffeomorphism can be written $\SU(2)^3$-equivariantly if we identify $S^3 \times \R^4$ with $\SU(2) \times \mathbb{H}$, equipped with the $\SU(2)^3$-action 
\begin{align} \label{action}
\left(q_1,q_2,q_3\right)\cdot(p,v)\mapsto (q_1 p\bar{q}_2,q_3v\bar{q}_1)
\end{align}
for $\left(q_1,q_2,q_3\right) \in \SU(2)^3$ and $(p,v) \in \SU(2) \times \mathbb{H}$. An equivariant diffeomorphism with respect to \eqref{action} can be written:
\begin{align*} 
 \SU(2)^2\times_{\Delta \SU(2)}\mathbb{H}  \to\SU(2) \times \mathbb{H}:\left[ p_1, p_2, v\right]\mapsto \left( p_1 \bar{p}_2, v \bar{p}_1 \right).
\end{align*} 

\begin{remark} Taking cyclic permutations of the factors of $\textup{SU}(2)^3$ in \eqref{action} gives three non-equivariantly isometric realisations of $\mathbf{S}(S^3)$. These arise naturally when considering the Bryant-Salamon metric as limits of the $\mathbb{B}_7$ and $\mathbb{D}_7$ families of $G_2$-metrics on $S^3 \times \mathbb{R}^4$.   
\end{remark}  

With the action of $\SU(2)^3$ explained, we will write down the $\SU(2)^3$-invariant $G_2$-structure on the space of principal orbits $\mathbb{R}_{>0} \times S^3 \times S^3$, and the induced $G_2$-holonomy metric. Here, we  $\SU(2)^3$-equivariantly identify $S^3 \times S^3$ with the principal orbit $\SU(2)^3 / \Delta_{1,2,3} \SU(2)$ via the inclusion map $\SU(2)^2 \times \lbrace 1 \rbrace \hookrightarrow \SU(2)^3$, where we view $\SU(2)^2 \times \lbrace 1 \rbrace$ acting on the left of $S^3 \times S^3 = \SU(2)^2$ in the obvious way, and $\lbrace 1 \rbrace \times \lbrace 1 \rbrace \times \SU(2)$ acting diagonally on the right\footnote{Since $\lbrace 1 \rbrace \times \lbrace 1 \rbrace \times \SU(2)$ acts trivially on the singular orbit, we can identify the singular orbit $\SU(2)^3 / \Delta_{1,2} \SU(2) \times \SU(2)$ with $\SU(2)^2 / \Delta \SU(2)$ in the same way.}. 

Denote by $e_1, e_2, e_3$, $e'_1, e'_2, e'_3$ a basis of left-invariant one-forms on $S^3 \times S^3 = \SU(2)^2$ such that   
\begin{align*}
d e_i = - e_j \wedge e_k& &d e'_i = - e'_j \wedge e'_k
\end{align*}
for $\left(i j k\right)$ any cyclic permutation of $\left( 1 2 3 \right)$. The diagonal right action of $\SU(2)$ on the space of left-invariant one-forms is via two copies of the adjoint representation $\mathfrak{su}_+(2) \oplus  \mathfrak{su}_-(2)$, where $\mathfrak{su}_\pm(2)$ is given by the linear span of the one-forms $e_i^\pm := \tfrac{1}{2} \left(e_i \pm e'_i \right)$ over $i = 1,2,3$.   

Up to discrete symmetries, which act by isometries on the induced $G_2$-holonomy metric, the $G_2$-structure inducing the Bryant-Salamon metric is given by the following closed form on $\mathbb{R}_{>0} \times S^3 \times S^3$
\begin{align} \label{eq:g2structure}
\varphi = p (e_1 \wedge e_2 \wedge e_3  - e'_1 \wedge e'_2 \wedge e'_3) + d \left( a(t) \sum_{i=1}^3e_i \wedge e'_i \right)     
\end{align} 
where $p>0$ is a constant representing the size of the cohomology class $\left[ \varphi \right]$, and $a(t)$ is a real-valued function of the geodesic variable $t \in \R_{\geq 0}$. Then $\varphi$ is co-closed if $a$ satisfies 
\begin{align} \label{hitchinfloweq0}
4\dot{a}^6 = 3a^4 - 8p a^3 + 6 p^2 a^2 - p^4.      
\end{align}
This has a unique solution with $a(0) = p$ for each $p$, extending smoothly over the singular orbit $S^3$.  

Using \cite[\S 5]{madsen:halfflat}, the metric $g = dt^2 + g_t$ on $\R_{>0} \times S^3 \times S^3$ induced by $\varphi$ can be written: 
\begin{align} \label{eq:g2metric}
g= dt^2 + A^2 \left(\sum_{i=1}^3 \left(e_i^+\right)^2 \right) + B^2 \left( \sum_{i=1}^3 \left(e_i^-\right)^2 \right)
\end{align}
for a pair of functions $\left(A,B\right)$, defined such that $A>0, B>0$ on $\R_{>0}$ and 
\begin{align*}
a = \frac{1}{8} \left( B^3 + B A^2 \right)& &p = \frac{1}{8} \left( B^3 - 3 B A^2 \right)
\end{align*}
The induced metric \eqref{eq:g2metric} has holonomy contained in $G_2$ if $\varphi$ is co-closed. By writing \eqref{hitchinfloweq0} in terms of the metric coefficients $\left(A,B\right)$, we see that this is equivalent to  $(A,B)$ satsifying
\begin{align} \label{hitchinfloweq}
\dot{A} = \frac{1}{2} \left( 1 - \frac{A^2}{B^2} \right)& & \dot{B} = \frac{A}{B}.   
\end{align}   
Clearly, taking $ A = \frac{1}{3} t$, $B = \frac{\sqrt{3}}{3} t$ is a solution to \eqref{hitchinfloweq} and this corresponds to the $\SU(2)^3$-invariant $G_2$-holonomy conical metric on $\R_{>0} \times S^3 \times S^3$. 

Up to scaling the resulting metric, and the diffeomorphism $t \mapsto \delta t$ rescaling the fibres of $\mathbf{S}(S^3)$ by a constant $\delta>0$, there is a unique solution to \eqref{hitchinfloweq} extending over the singular orbit $S^3 = \SU(2)^2 / \Delta \SU(2)$. This can be written explicitly (see \cite[Section 2.2.1]{goncalog2}) in terms of the variable $r(t) = \sqrt{3} B(t) $, $r \in \left[ 1, \infty \right)$ as 
\begin{align} \label{eq:g2sols}
A(r) = \tfrac{r}{3} \sqrt{1- r^{-3}}& &B(r) = \tfrac{r}{\sqrt{3}}. 
\end{align}
The asymptotic model for the geometry of this metric is the co-homogeneity one $G_2$-cone over the nearly-K\"{a}hler $S^3 \times S^3 = \SU(2)^3 / \Delta_{1,2,3} \SU(2)$. Outside the singular orbit, we can identify $\mathbf{S}(S^3)$ with the smooth manifold $\R_{>0} \times S^3 \times S^3$, and the complete $G_2$-metric $g$ given by \eqref{eq:g2metric}, with $\left(A,B \right)$ defined by \eqref{eq:g2sols}, satisfies \[\vert g - g_C \vert = O(t^{-3})\] as $t\rightarrow \infty$, where $t$ denotes the radial parameter on the cone and we take norms with respect to the conical metric $g_C$.

\section{$G_2$ Gauge Theory} \label{section:G2overview} 
Let $\left(M^7,\varphi \right)$ be a $G_2$-manifold, equipped with a principal $G$-bundle $P \rightarrow M$ for a compact, semi-simple Lie group $G$, where $\varphi \in \Omega^3 (M)$ is a torsion-free $G_2$-structure on $M$. Recall that a connection $A$ on $P$ is called a $G_2$-\textit{instanton} if it satisfies the $G_2$-\textit{instanton equations}: 
\begin{align} \label{eq:g2instanton0}
F_A \wedge {*} \varphi = 0. 
\end{align}
We will write \eqref{eq:g2instanton0} in the case where $\left( M, \varphi \right)$ is foliated by parallel hyper-surfaces. Recall from \S \ref{section:G2define} that we can define a $G_2$-structure on $M$ 
\begin{align*} 
\varphi = dt \wedge \omega + \Reom& &*\varphi = - dt \wedge \Imom + \tfrac{1}{2} \omega^2
\end{align*} 
in terms of a one-parameter family of half-flat structures $\left( \omega, \Omega \right)_{t \in I}$. Written in temporal gauge $A=A_t$, the $G_2$-instanton equations are
\begin{subequations} \label{g2instanton}
\begin{align} 
F_{A_t} \wedge \omega^2 =0, \label{g2inst1}\\
F_{A_t} \wedge \Imom - \frac{1}{2} \partial_t A_t \wedge \omega^2 = 0. \label{g2inst2} 
\end{align}
\end{subequations}
If the $G_2$-structure $\varphi$ is torsion-free, then $\left( \omega, \Omega \right)_{t \in I}$ is subject to the evolution equation \[d \Imom = - \tfrac{1}{2} \partial_t \left( \omega^2 \right).\] Thus the $G_2$-instanton equation \eqref{g2inst1} is preserved under evolution by \eqref{g2inst2}.  

If $\left(M,\varphi\right)$ is a $G_2$ cone, i.e. $\varphi=\varphi_C$ is given by \eqref{G2cone}, then scale invariant solutions to \eqref{g2instanton} are pulled back from solutions to the Hermitian Yang-Mills equations on the link 
\begin{align} \label{eq:NKinstanton}
F_{A} \wedge \Reom^{nK} = 0& &F_{A} \wedge \left(\omega^{nK}\right)^2 =0. 
\end{align}
Solutions of \eqref{eq:NKinstanton} are referred to as \textit{nearly-K\"{a}hler instantons} and appear naturally as asymptotic limits of $G_2$-instantons on asymptotically conical $G_2$-manifolds (cf. \cite{nKinstantons:harland}).

\subsection{Invariant Instanton ODEs} \label{section:g2instantonodes}

In the invariant setting, we consider the $\SU(2)^3$-invariant co-homogeneity one $G_2$-metrics on the spinor bundle $\mathbf{S}(S^3)$ described in \S \ref{section:G2define}. Following \cite{goncalog2}, we assume the bundle and connection form are also invariant under some lift of the $\SU(2)^3$-action to the total space of the bundle.  

$\SU(2)^3$-homogeneous bundles over the principal orbit $\SU(2)^3 / \Delta \SU(2)$ of $\mathbf{S}(S^3)$ with gauge group $\SU(2)$ are classified by homomorphisms $\Delta \SU(2) \ra \SU(2)$. This gives exactly two non-equivariantly equivalent bundles over the principal orbit, which are both trivial bundles topologically but only one is equivariantly trivial. They are defined by the trivial homomorphism $\Delta \SU(2) \ra \SU(2)$ and the identity homomorphism $\Delta \SU(2) \ra \SU(2)$, and are of the form
\begin{align*}
P=\SU(2)^3\times_{\Delta  \SU(2)}  \SU(2).
\end{align*}
Wang's Theorem \cite[Thm. A]{wang1958} tells us that the space of invariant connections on these homogeneous bundles is an affine space of intertwiners of the $\Delta \SU(2)$-action on left-invariant one-forms on $\SU(2)^3 / \Delta \SU(2)$ and on the Lie-algebra of the gauge group. Recall also from \S \ref{section:G2define} that the action of $\Delta \SU(2)$ acts on the tangent space of $\SU(2)^3 / \Delta \SU(2)= \SU(2)^2$ as two copies of the adjoint representation $\mathfrak{su}_+(2) \oplus \mathfrak{su}_- (2)$. 

Since $\Delta \SU(2)$ acts trivially on the gauge group for the trivial homogeneous bundle, the only $\SU(2)^3$-invariant connection on this bundle is the trivial flat connection. Meanwhile, for the non-trivial homogeneous bundle, $\Delta \SU(2)$ acts on $\mathfrak{su}(2)$ via the adjoint representation, thus the space of invariant connections is two-dimensional by Schur's Lemma. 

Using this description of invariant connections, and the description of $\SU(2)^3$-invariant torsion-free $G_2$-structures \eqref{eq:g2structure} in \S \ref{section:G2define}, Lotay and Oliveira \cite[Prop. 5]{goncalog2} write the $G_2$-instanton equations \eqref{g2instanton} as in the following proposition. 
\begin{prop}[\cite{goncalog2}]\label{prop:g2instode} On $\mathbb{R}_{>0} \times S^3 \times S^3$ with a $G_2$-structure given by \eqref{eq:g2structure}, $\textup{SU}(2)^3$-invariant instantons can be written, up to gauge transformation, as
\begin{align} A_t = f_+ \left( \sum_{i=1}^3 E_i \otimes e_i^+ \right) +  f_- \left(  \sum_{i=1}^3 E_i \otimes e_i^- \right)
\end{align}
with $\lbrace E_i \rbrace_{i=1,2,3}$ a basis of left-invariant vector-fields dual to $\lbrace e_i \rbrace_{i=1,2,3}$, and real-valued functions $\left( f_+, f_- \right)$ satisfying the ODE system  
\begin{align} \label{eq:g2instode}
\dot{f}_+ = \frac{f_+}{A}\left(1- \frac{A^2}{B^2} - f_+ \right) + f_-^2 \frac{A}{B^2}& &\dot{f}_-= \frac{2f_-}{A} \left( f_+ - 1 \right). 
\end{align}
\end{prop}
Here, we can identify both $\SU(2)^3$-homogeneous bundles with the trivial $\SU(2)^2$-homogeneous bundle over $S^3 \times S^3$, up to $\SU(2)^2$-equivariant isomorphism. We can recover the case of the flat connection on the trivial $\SU(2)^3$-homogeneous bundle by taking $f_+ = f_- =0$. 

We note in the following lemma that there is an additional discrete symmetry of \eqref{eq:g2instode}, a pull-back of the non-equivariant isometry exchanging the factors of $S^3 \times S^3$ on the principal orbits.
\begin{lemma} \label{symmetry:g2instode} The transformation $\left(f_+,f_- \right) \ra \left( f_+, - f_- \right)$ is a symmetry of \eqref{eq:g2instode}. 
\end{lemma}
Before stating the main theorem, we will study the behaviour of \eqref{eq:g2instode} in the two limits $t\ra 0$, $t\ra \infty$. Firstly, up to $\SU(2)^2$-equivariant isomorphism, we can extend the trivial bundle over the principal orbit $S^3 \times S^3$ to the singular orbit $S^3 = \SU(2)^2 / \Delta \SU(2)$ at $t=0$ in one of two ways: using either the identity homomorphism $\Delta \SU(2) \ra \SU(2)$ or the trivial homomorphism\footnote{These $\SU(2)^2$-equivariant bundles are referred to, respectively, as $P_\mathrm{Id}$ and $P_1$ in \cite{goncalog2}.}. As is shown in \cite{goncalog2}, each extension gives a one-parameter family of solutions to \eqref{eq:g2instode} near $t=0$; we restate their results in the following proposition.
 \begin{prop}[\cite{goncalog2}] \label{prop:G2localsol} In a neighbourhood of the singular orbit at $t=0$, solutions to \eqref{eq:g2instode} are in two one-parameter families $T_\gamma$, $T'_{\gamma'}$ for parameters $\gamma, \gamma' \in \mathbb{R}$:
 \begin{enumerate}
 \item The family $T_\gamma$ extends over the trivial $\SU(2)^2$-homogeneous bundle over the singular orbit, and these solutions satisfy, in a neighbourhood of $t=0$,
  \begin{align}  \label{eq:P0localsol}
f_+ = \gamma t^2 + O(t^4),&  &f_- = 0.
 \end{align}
 \item  The family $T'_{\gamma'}$ extends over the non-trivial $\SU(2)^2$-homogeneous bundle over the singular orbit, and these solutions satisfy, in a neighbourhood of $t=0$
  \begin{align} \label{eq:Pidlocalsol}
f_+ = 1 + O(t^2),&  &f_- = \gamma' + O(t^2). 
 \end{align} 
 \end{enumerate}
 \end{prop}
\begin{remark} For later use, we compute some additional terms in the Taylor series of $T'_{\gamma'}$ near $t=0$:
\begin{align}
f_+ = 1 + \frac{3}{8} \left( \gamma'^2 -1 \right) t^2 + O(t^4)& &f_- = \gamma' + \frac{3}{4} \left( \gamma'^2 -1 \right) \gamma' t^2 + O(t^4).
\end{align}
\end{remark} 
The $G_2$-structure on $\mathbf{S}(S^3)$ is asymptotic to the conical $\SU(2)^3$-invariant $G_2$-structure over $S^3 \times S^3$; more precisely $A = \tfrac{t}{3} + O(t^{-2})$, $B = \tfrac{t}{\sqrt{3}} + O(t^{-2})$ for $t$ sufficiently large. So if $\left(f_+,f_-\right)$ are bounded a-priori, the system \eqref{eq:g2instode} differs from the corresponding instanton equations on the cone
\begin{align} \label{eq:G2odecone}
\dot{f}_+ = \frac{1}{t}\left(2 f_+ - 3 f_+^2 + f_-^2 \right),& &\dot{f}_-= \frac{6}{t} f_- \left( f_+ - 1 \right) 
\end{align}
by $O(t^{-4})$ terms.
 
We note here that, as well as symmetry of Lemma \ref{symmetry:g2instode}, there are additional discrete symmetries of the conical equations \eqref{eq:G2odecone}, coming from the non-equivariant isometry of permuting the three copies of $\SU(2)$ in the $G_2$-cone metric over $S^3 \times S^3 = \SU(2)^3/  \Delta \SU(2)$. A detailed exposition of the symmetries of $S^3\times S^3$ can be found in \cite{AW01} and the following lemma describes the pull-back of these symmetries to the conical equations more precisely.
\begin{lemma} \label{lemma:conesymmetries} The permutation group on $3$ elements acts by symmetries on \eqref{eq:G2odecone}. Up to a change of parametrisation $\left( f_+, f_- \right) \mapsto \left( f_+ + \tfrac{2}{3}, \sqrt{3} f_- \right)$,  this action is generated by the transformations 
\begin{align} \label{eq:conesymmetries} \begin{pmatrix} f_+ \\
f_- \\ 
\end{pmatrix} \mapsto \frac{1}{2}\begin{pmatrix} -1 & \sqrt{3} \\
-\sqrt{3} & -1 \\ 
\end{pmatrix} \begin{pmatrix} f_+ \\
f_- \\ 
\end{pmatrix} & \ \ \ \ \text{and} \ \ \ \ \begin{pmatrix} f_+ \\
f_- \\ 
\end{pmatrix} \mapsto \begin{pmatrix} f_+ \\
-f_- 
\end{pmatrix}. 
\end{align}    
\end{lemma}
\begin{proof} Note that the transformations \eqref{eq:conesymmetries} above are a clockwise rotation about the origin by $\frac{2\pi}{3}$ and a reflection in the plane across the line $\{f_-=0\}$. The resulting symmetries can be seen from the re-parametrised system $\left( f_+, f_- \right) \mapsto \left( f_+ + \tfrac{2}{3}, \sqrt{3} f_- \right)$ 
\begin{align} \label{eq:aut}
\dot{f}_+ = \frac{1}{t}\left( 3 f_-^2 - f_+ \left( 3 f_+ + 2 \right) \right)& &\dot{f}_-= \frac{2}{t} f_- \left( 3 f_+ - 1 \right) .
\end{align}
The phase diagram of this system is shown in Figure \ref{fig:conePhase} and it pictorially demonstrates the $S_3$-symmetry of the system. 
\end{proof} 

\begin{figure}
\includegraphics[scale=0.5]{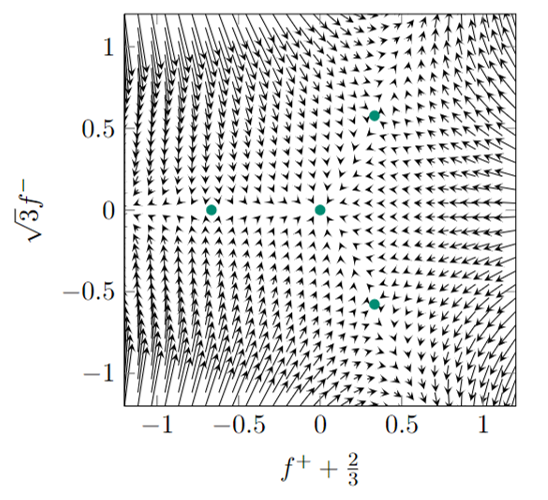}
\caption{The phase diagram of the autonomous system \eqref{eq:aut}, with the four critical points marked; the symmetry described in Lemma \ref{lemma:conesymmetries} is evident here.}\label{fig:conePhase}
\end{figure}

Furthermore, we will see later that any bounded solution of the full system \eqref{eq:g2instode} will converge to one of the critical points of \eqref{eq:G2odecone}, namely
\begin{align*}
\left(0,0\right)& &\left(1,\pm1\right)&  &(\tfrac{2}{3},0).
\end{align*}
These critical points correspond to $\SU(2)^3$-invariant nearly-K\"{a}hler instantons: the points $\left(0,0\right)$, $\left(1,\pm1\right)$ are all the flat connection $A^\flat$ in different non-equivariant gauges, while the only non-trivial instanton is $A^{nK}$, given by $(f_+,f_-)=(\tfrac{2}{3},0)$. Note that $A^{nK}$ is the canonical connection on the non-trivial homogeneous bundle $\SU(2)^3 \ra \SU(2)^3 / \Delta \SU(2)$ over $S^3 \times S^3$, and has been studied previously as a nearly-K\"{a}hler instanton in \cite{nKinstantons:harland}. 

As it will be useful in understanding the deformation theory of $G_2$-instantons converging to $A^{nK}$, we prove the following lemma. 

\begin{lemma} \label{lemma:decayrate} Solutions to \eqref{eq:G2odecone} converging to the asymptotically stable critical point $(\tfrac{2}{3},0)$ are in a two-parameter family $\mu, \nu \in \mathbb{R}$ for $t$ sufficiently large:
\begin{align} \label{eq:asymptotic}
f_+ = \tfrac{2}{3} + \mu t^{-2} + O(t^{-3})& &f_- = \nu t^{-2} + O(t^{-3})    
\end{align}
\end{lemma}
\begin{proof}
By reparametrising \eqref{eq:G2odecone} by $t \mapsto e^{t}$, we obtain a homogeneous system of the form $\dot{y} = M (y)$. The linearisation at $y_0=(\tfrac{2}{3},0)$ of this system has a repeated eigenvalue $-2$, and hence we can find a 2-parameter family of solutions, as given in the statement of the lemma.
\end{proof} 
\subsection{Complete Solutions} \label{section:completesolutions}
With the two limiting behaviours $t\ra 0$, $t \ra \infty$ understood, we now discuss complete solutions to the ODE system \eqref{eq:g2instode}. The family of local solutions  $T_\gamma$ in Proposition \ref{prop:G2localsol} can be obtained explicitly by solving  \eqref{eq:g2instode} with $f_- = 0$, and was found previously in \cite{clarke}. In terms of the variable $r(t) = \sqrt{3} B(t)$ in \eqref{eq:g2sols}, these solutions are given by 
\begin{align}\label{eq:clarke}
f_+ = \frac{2}{3} \left( 1 + \frac{2 \gamma ( r - 1 ) - 3 r}{2 \gamma r (r^2 -1) + 3r} \right),& &f_-=0.
\end{align}
Clearly, these solutions exist for all time if and only if $\gamma\geq0$, and $\gamma =0$ is just the flat connection $A^\flat$ at $\left(0,0\right)$. Furthermore, in the limit $\gamma \ra \infty$, the solution \eqref{eq:clarke} converges outside the singular orbit at $r=1$ to another explicit solution of \eqref{eq:g2instode}, namely
\begin{align}\label{eq:goncalog2}
f_+ = \tfrac{2}{3} \left( 1 + \tfrac{1}{r ( r+ 1)} \right),& &f_-=0.
\end{align}

The limiting solution \eqref{eq:goncalog2} still extends over the singular orbit, but on a different invariant bundle\footnote{See \cite[Theorem 2]{goncalog2} for an explanation of this limit in terms of the bubbling and removable-singularity phenomenon found in \cite{tian2000gauge}.}, as the member of the family $T'_{\gamma'}$ with $\gamma'=0$. This solution was found previously in  \cite{goncalog2}, but we now show in the following theorem that it lies in a one-parameter family of solutions with $\gamma'$ non-zero.      

\begin{theorem} \label{thm:g2inst} $\SU(2)^3$-invariant instantons on $\mathbf{S}(S^3)$ with gauge group $\SU(2)$, and quadratic curvature decay i.e. norm of the curvature $| F | = O(t^{-2}) $ with respect to the cone metric, are in two one-parameter families $T_\gamma$, $\gamma \in \left[ 0, \infty \right)$, and $T'_{\gamma'}$, $\gamma' \in \left[-1, 1 \right]$. Moreover
\vspace{-0.3cm}
\begin{enumerate}
\item The isometry exchanging the factors of $\SU(2)^2$ on the principal orbits sends $T'_{\gamma'} \mapsto T'_{-\gamma'}$;  
\item $T_0$, $T'_1$, $T'_{-1}$ are flat, otherwise $T_\gamma$, $T'_{\gamma'}$ are irreducible;
\item In the gauge given in Proposition \ref{prop:g2instode}, the irreducible $T_\gamma$, $T'_{\gamma'}$ are asymptotic to $A^{nK}$ with rate $-3$, i.e. $| A - A^{nK} | = O( t^{-3})$ for $A= T_\gamma$, $T'_{\gamma'}$, where we take norms with respect to the cone metric.     
\end{enumerate}
\end{theorem}

\begin{remark} We note there is an error in \cite[Prop. 5]{goncalog2}, which claims a faster rate of convergence for $T'_{0}$.
\end{remark}

Before proving this theorem, we will say a few words about the quadratic curvature decay condition, and the asymptotic convergence condition $| A - A^{nK} | = O( t^{-3})$ in terms of the ODE system \eqref{eq:g2instode}. We note that $1$-forms $e^\pm_i$ on the link of the cone satisfy $| e^\pm_i | = O(t^{-1})$ with respect to the cone metric, so $| A - A^{nK} | = O( t^{-3})$ if and only if
\begin{align*}
f_+ = \tfrac{2}{3} + O(t^{-2}),& &f_- =  O(t^{-2}).
\end{align*}
The curvature of the connection decaying quadratically can be read off using Proposition \ref{prop:g2instode} and the expression $F_A = F_{A_t} + \dot{A}_t \wedge dt $ for curvature in the temporal gauge $A = A_t$: it is equivalent to solutions $(f_+, f_-)$ of \eqref{eq:g2instode} being bounded.   

With this said, the analysis for the family $T_\gamma$ follows from its explicit form \eqref{eq:clarke}, and the transformation $T'_{\gamma'} \mapsto T'_{-\gamma'}$ is not hard to see from applying Lemma \ref{symmetry:g2instode} to the local expression \eqref{eq:Pidlocalsol} for $T'_{\gamma'}$. For the rest of this section, we will prove Theorem \ref{thm:g2inst} by showing that the local solutions $T'_{\gamma'}$ exist for all time if  $\gamma' \in \left[-1, 1 \right]$, are asymptotically of the form \eqref{eq:asymptotic} if  $\gamma' \in \left(-1, 1 \right)$, and otherwise cannot be bounded. 

The strategy will involve constructing sets that are forward-invariant under evolution by the ODE system \eqref{eq:g2instanton0} and that contain our short-time solutions in Proposition \ref{prop:G2localsol}. Once we have this, we will use the asymptotic description of this system \eqref{eq:G2odecone0} to determine the long-time behaviour of the solutions lying in these invariant sets.

\begin{lemma} The following sets are forward-invariant for \eqref{eq:g2instode}:
\begin{enumerate}[\normalfont(i)]
\item $H_\pm := \lbrace \left(f_+, f_- \right) \in \mathbb{R}^2 \mid \pm f_->0 \rbrace$;
\item $\mathcal{R}_\infty := \lbrace \left(f_+, f_-\right) \in \mathbb{R}^2 \mid f_+>1, f_->1 \rbrace$;
\item  $\mathcal{R}_0 := \lbrace  \left(f_+, f_-\right) \in \mathbb{R}^2 \mid \tfrac{2}{3} < f_+ < 1, 0<f_-<1 \rbrace$.
\end{enumerate} 
\end{lemma}
\begin{proof}
 \begin{enumerate}[(i)] 
 \item As previously mentioned, setting $f_- = 0$ gives a family of solutions to \eqref{eq:g2instode}. Hence, by symmetry of Lemma \ref{symmetry:g2instode}, we will reduce to the case $f_->0$ in what follows. 
 \item For $f_- >0$, the sign of $\dot{f}_-$ is given by the sign of $f_+ - 1$, hence a solution cannot leave $\mathcal{R}_\infty$ via the line $f_+>1$, $f_-=1$. Secondly, $\left. \dot{f}_+ \right|_{f_+ =1} =  \tfrac{A}{B^2} \left( f_-^2 -1 \right)$, hence a solution cannot leave via the line $f_+=1$, $f_->1$ either. Finally, the intersection $f_+=f_-=1$ is a critical point of \eqref{eq:g2instode}, corresponding to the flat connection. 
  \item By part (i), we can always assume $f_->0$. Using the same argument as part (ii), we see that $\dot{f}_- < 0 $ when $1>f_+>0$, $\left. \dot{f}_+ \right|_{f_+ =1}<0$ when $1>f_->0$, and $f_+ = f_- = 1$ is a critical point. Thus, it only remains to show a solution cannot leave $\mathcal{R}_0$ via the line segment $f_+=\tfrac{2}{3}$, $f_->0$. This follows from the inequality, $3 A^2 < B^2$ on $t>0$, which can easily be seen from \eqref{eq:g2sols}. With this inequality, it is clear that
\begin{align*}
\left. \dot{f}_+ \right|_{f_+ = \tfrac{2}{3}} =  \tfrac{2}{3 A} \left( \tfrac{1}{3} - \tfrac{A^2}{B^2} \right) +  f_-^2 \tfrac{A}{B^2} > 0.
\end{align*}
 \end{enumerate}
 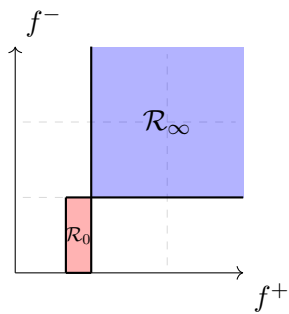
\begin{figure}[h]
\begin{center}
\begin{tikzpicture}[]
\draw[help lines, color=gray!30, dashed] (0.1,0.1) grid (2.9,2.9);
\draw[->] (0,0)--(3,0) node[below right]{$f^+$};
\draw[->] (0,0)--(0,3) node[above right]{$f^-$};

\fill [fill=blue, opacity=0.3] (1,1) rectangle (3,3);

\fill [fill=red, opacity=0.3] (2/3,0) rectangle (1,1);

\draw[thick] (2/3,0) -- (1,0);
\draw[thick] (2/3,0) -- (2/3,1);
\draw[thick] (1,0) -- (1,3);
\draw[thick] (2/3,1) -- (3,1);

\node[scale=0.7] at (0.83,0.5) {$\mathcal{R}_0$};
\node at (2,2) {$\mathcal{R}_\infty$};
\end{tikzpicture}
\end{center}
\caption{The invariant sets $\mathcal{R}_0$ and $\mathcal{R}_\infty$}
\end{figure}
\end{proof}
\begin{lemma}  
A solution $\left(f_+,f_- \right)$ to \eqref{eq:g2instode} lying in $\mathcal{R}_\infty$ at some initial time $t_0>0$, cannot be uniformly bounded for all $t \geq t_0$.
\end{lemma}
\begin{proof} Since $f_-$ is strictly increasing in $\mathcal{R}_\infty$, if the solution $\left(f_+,f_- \right)$ blows up at finite time $T$, then necessarily the solution cannot be bounded for all $t<T$. On the other hand, recalling the asymptotic behaviour \eqref{eq:G2odecone} of the system, if we re-parametrise \eqref{eq:g2instode} by $t \mapsto e^t$, then for $t$ sufficiently large and $\left(f_+,f_- \right)$ lying in a compact subset of $\mathcal{R}_\infty$, this re-parametrised system is asymptotic to the autonomous system
\begin{align} \label{eq:G2odecone0}
\dot{f}_+ = 2 f_+ - 3 f_+^2 + f_-^2& &\dot{f}_- = 6 f_-\left( f_+ - 1 \right) 
\end{align}
up to terms decaying exponentially in $t$. The theory of non-autonomous systems asymptotic to autonomous systems can be found in \cite{markusdiffsystem}; here, we apply \cite[Thm.3]{markusdiffsystem}, which says that if solutions to \eqref{eq:G2odecone0} in $\mathcal{R}_\infty$ cannot be uniformly bounded for sufficiently large times, then neither can solutions to \eqref{eq:g2instode}.  

Assume for a contradiction that a solution to \eqref{eq:G2odecone0} exists for all time in $\mathcal{R}_\infty$, and is uniformly bounded. Since $f_-$ is monotonically increasing in $\mathcal{R}_\infty$, there exists an $\epsilon>0$ such that $ f_- > 1 + \epsilon$ for $t>t_0$. If we let $f_+(\epsilon)>1$ be the unique solution to $2 f_+ - 3 f_+^2 + (1+ \epsilon)^2 =0$ in $\mathcal{R}_\infty$, then $f_+$ is strictly increasing in $1<f_+< f_+( \epsilon)$ for time $t>t_0$, and hence $f_+$ is uniformly bounded below away from $1$. But this is a contradiction, since it implies $\dot{f}_-$ is bounded below away from zero, and hence $f_-$ cannot be bounded.  
\end{proof}
\begin{lemma}  
A solution  $\left(f_+,f_- \right)$ to \eqref{eq:g2instode} lying in $\mathcal{R}_0$ at some initial time $t_0>0$ converges to $A^{nK} = \left( \tfrac{2}{3}, 0 \right)$, such that as $t\ra \infty$
\begin{align*}
f_+ = \tfrac{2}{3} + \mu t^{-2} + O(t^{-3}),& &f_- = \nu t^{-2} + O(t^{-3})
\end{align*}
for some $\mu, \nu$ non-zero. 
\end{lemma}
\begin{proof} The key to proving this statement will be to show that a solution in $\mathcal{R}_0$ must get arbitrarily close to the critical point $\left( \tfrac{2}{3}, 0 \right)$ of  \eqref{eq:G2odecone0} at \textit{some} forward time. Once we have proved this, we can apply \cite[Thm.2]{markusdiffsystem}; since the linearisation of \eqref{eq:G2odecone0} near $\left( \tfrac{2}{3}, 0 \right)$ has only (real) negative eigenvalues, it is asymptotically stable for \eqref{eq:g2instode}. 

Moreover, recall from \S \ref{section:g2instantonodes} that \eqref{eq:g2instode} differs from the cone equations \eqref{eq:G2odecone} by $O(t^{-4})$ terms, so we recover the asymptotic form of a solution $\left(f_+, f_- \right)$ converging to $\left( \tfrac{2}{3}, 0 \right)$ up to $O(t^{-3})$ using the asymptotic form of a solution on the cone in Lemma \ref{lemma:decayrate}. 

Let $\left( f_+, f_- \right)$ be a solution to the re-parametrisation $t \mapsto e^t$ of \eqref{eq:g2instode} which lies in $\mathcal{R}_0$. Since $f_-$ is strictly decreasing, there must be an $\epsilon \in (0,1)$ such that $f_- < 1-\epsilon$ for all forward time. Then we can take $T(\epsilon)>t_0$ sufficiently large such that $\dot{f}_+ < 2 f_+ - 3 f_+^2 + f_-^2 + \epsilon$ for all $t>T$, and an $f_+(\epsilon)$ sufficiently close to $1$ such that $\dot{f}_+ < 2 f_+ - 3 f_+^2 + \left(1-\epsilon\right)^2 + \epsilon <0$ on $f_+(\epsilon) <f_+ <1$, $t>T$. Hence, we can bound $f_+$ away from $1$ for $t>T$. 

On the other hand, $f_-\left( f_+ - 1 \right)$ cannot be bounded above away from zero, since this would imply that $\dot{f}_-$ would be bounded above away from zero after some sufficiently large time, and hence $f_-$ would be unbounded. Combined with the previous observation, this implies $f_-$ cannot be bounded away from $0$, and hence $f_- \ra 0$ as $t\ra \infty$ since $f_-$ is decreasing. Similarly, $|2 f_+ - 3 f_+^2 + f_-^2|$ cannot be bounded below away from $0$, and hence $f_+$ cannot be bounded away from $\tfrac{2}{3}$, and we are done. 
\end{proof}

We finally consider the local power-series solutions $T'_{\gamma'}=\left(f_+, f_- \right)_{\gamma'}$ of \eqref{eq:Pidlocalsol}. The solution with $\gamma'=0$ is the explicit solution \eqref{eq:goncalog2}, and one can take $\gamma' > 0$ otherwise, up to the symmetry of Lemma \ref{symmetry:g2instode}. Then $\left(f_+, f_- \right)_{\gamma'} \in \mathcal{R}_0$ when $0<\gamma'<1$, $\left(f_+, f_- \right)_{\gamma'} \in \mathcal{R}_\infty$ when $1<\gamma'$, and $T'_{\gamma'}$ with $\gamma'=1$ is the critical point $(1,1)$ of \eqref{eq:g2instode} corresponding to the flat connection. This completes the proof of Theorem \ref{thm:g2inst}

\begin{remark} We can understand the limits $\gamma \ra 0$, $\gamma' \ra \pm 1$, as the curvature of the connections $T_\gamma$, $T'_{\gamma'}$ vanish, in terms of $G_2$-instantons on the asymptotic cone. In terms of solutions to the ODE system \eqref{eq:g2instode}, close to this limit, the trajectories $\lbrace \left( f_+, f_- \right)(t) \mid t\in \R_{\geq 0} \rbrace$ of $T_\gamma$, $T'_{\gamma'}$ are modelled on instantons on the cone after some sufficiently large time. 

These limiting solutions of \eqref{eq:G2odecone} are straight line segments in the plane, interpolating between the flat connections $\left(0,0\right), \left(1,1\right), \left(1,-1\right)$ as  $t\ra 0$ and the nearly-K\"{a}hler instanton $A^{nK}:=(\tfrac{2}{3},0)$ as $t\ra \infty$, and are given explicitly given by
\begin{align*}
f_+ = \tfrac{2t^2}{1+3t^2},& &f_-=0
\end{align*}
and its image under the symmetries of Lemma \ref{lemma:conesymmetries}. 
\end{remark} 
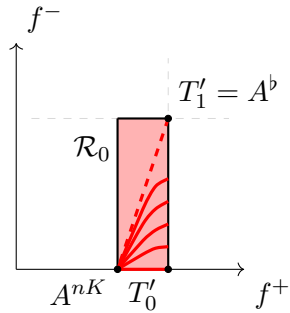
\begin{figure}[h]
\begin{center}
\begin{tikzpicture}[scale=2]
\draw[help lines, color=gray!30, dashed] (0.1,0.1) grid (1.4,1.4);
\draw[->](0,0)--(1.5,0) node[below right]{$f^+$};
\draw[->](0,0)--(0,1.5) node[above right]{$f^-$};

\fill[fill=red, opacity=0.3] (2/3,0) rectangle (1,1);

\draw[thick] (2/3,0) -- (1,0);
\draw[thick] (2/3,0) -- (2/3,1);
\draw[thick] (1,0) -- (1,1);
\draw[thick] (2/3,1) -- (1,1);

\node at (0.5,0.8) {$\mathcal{R}_0$};

\draw[very thick,red] (1,0) -- (2/3,0);
\draw[very thick, dashed,red] (1,1) -- (2/3,0);
\draw[very thick,red] plot [smooth] coordinates {(1,0.6) (11/12,0.55) (5/6,0.4) (2/3,0)};
\draw[very thick,red] plot [smooth] coordinates {(1,0.45) (11/12,0.4) (5/6,0.3) (2/3,0)};
\draw[very thick,red] plot [smooth] coordinates {(1,0.3) (11/12,0.26) (5/6,0.19) (2/3,0)};
\draw[very thick,red] plot [smooth] coordinates {(1,0.15) (11/12,0.14) (5/6,0.10) (2/3,0)};
\filldraw[black] (2/3,0) circle (0.6pt) node[below left]{$A^{nK}$};
\filldraw[black] (1,1) circle (0.6pt) node[above right]{$T'_1=A^{\flat}$};
\filldraw[black] (1,0) circle (0.6pt) node[below left]{$T'_0$};
\end{tikzpicture}
\end{center}
\caption{The trajectories of the family of solutions $T'_{\gamma'}$, $\gamma' \in \left[0,1\right]$. The limiting trajectory is the straight line interpolating between $A^\flat$ and $A^{nK}$.}
\end{figure}
\section{Uniqueness of Unobstructed Instantons} \label{section:deform} 

In the previous section, we classified $\SU(2)^3$-invariant solutions to the $G_2$-instanton equations, giving two families asymptotic to the non-trivial invariant nearly-K\"{a}hler instanton $A^{nK}$ on $S^3 \times S^3$. One might then hope to produce more examples of $G_2$-instantons on the Bryant-Salamon metric by considering deformations of these symmetric solutions away from the symmetric regime. 

However, using the deformation theory of $G_2$-instantons on asymptotically conical $G_2$-manifolds worked out in \cite{driscollthesis}, we will find that these invariant families actually classify all $G_2$-instantons on $\mathbf{S}(S^3)$ asymptotic to $A^{nK}$, at least if their deformations are unobstructed. This essentially follows ideas from \cite{driscollthesis}, but for completeness, we will first briefly recount the required theory, following \cite{driscollthesis} and \cite{nakajima}. 
\subsection{Deformation Theory}
Let $\left(M^7, \varphi \right)$ be an AC $G_2$-manifold, with asymptotic cone $C(\Sigma)$, and $P\ra M$ be a principal $G$-bundle with $G$ compact, semi-simple. Extending the radial parameter on $C(\Sigma) \cong \R_{>0} \times \Sigma$ to a smooth positive function $t$ on $M$, we define the weighted norms for smooth compactly-supported adjoint-valued $p$-forms $\Phi \in \Omega_c^p \left( \mathrm{ad} P \right)$:
\begin{align*}
|| \Phi ||_{W^{k,2}_\mu}: = \left( \sum^k_{j=0} \int_M \left| t^{j-\mu} \nabla_A^j \Phi \right|^2 t^{-7} \right)^{\tfrac{1}{2}}& &|| \Phi ||_{C^k_\mu} := \sum^k_{j=0} \sup_M \left| t^{j-\mu} \nabla_A^j \Phi \right| 
\end{align*} 
for some $\mu<0$, and a fixed connection $A$ on $P$. 

We will use $\Omega^p_{k,\mu} \left( \mathrm{ad} P \right)$ to denote the completion of $\Omega_c^p \left( \mathrm{ad} P \right)$ with respect to the weighted Sobolev norm $W^{k,2}_\mu$, and define
\begin{align*}
\Omega^p_\mu \left( \mathrm{ad} P \right):= \cap_{k\geq 0} \Omega^p_{k,\mu} \left( \mathrm{ad} P \right)
\end{align*}
A weighted version of the standard Sobolev embedding in dimension seven \cite[Thm.2.5.5]{driscollthesis} can be used to show that $\Phi \in \Omega_\mu^p \left( \mathrm{ad} P \right)$ implies that $|| \Phi ||_{C^k_\mu}< \infty$ for all $k\geq 0$, i.e. $| \nabla_A^j \Phi | = O\left(t^{\mu-j} \right)$. 

To consider the space of connections on $P$ with fixed asymptotic behaviour, we fix a \textit{framing at infinity}: a pair $\left(P_\infty, A_\infty \right)$ consisting of a bundle $P_\infty \ra \Sigma$ equipped with a connection $A_\infty$, such that $P \ra M$ is identified with $P_\infty$ pulled back over the conical end of $M$. We will define a connection $A$ on $P$ as \textit{asymptotic to} $A_\infty$ \textit{at polynomial rate} $\mu<0$ if $||A- A_\infty||< \infty$ for all $k\geq 0$, where we pull back $A_\infty$ to the end of $M$ and use the ${W^{k,2}_\mu}$-norm defined using the covariant derivative associated to $A_\infty$. 

The relevant space of connections we will consider is the affine space $\mathcal{A}_{\mu -1}$, $\mu<0$ of all connections asymptotic to $A_\infty$ with polynomial rate strictly less than $-1$. The correct notion of gauge equivalence of two connections in $\mathcal{A}_{\mu -1}$ is to use the subgroup $\mathcal{G}_{\mu}$ of \textit{framed gauge transformations with weight} $\mu$: gauge transformations of $P$ which are asymptotic to the identity on $P_\infty$ at rate $\mu$, see \cite{nakajima}, \cite{driscollthesis} for precise details of how to set-up these weights. The property of the gauge group $\mathcal{G}_\mu$ we will use here is that the tangent space to the $\mathcal{G}_\mu$-orbit through some $A \in  \mathcal{A}_{\mu -1}$ is spanned by elements of the form $d_A \Phi$ for some $\Phi \in \Omega_{\mu}^0$.

Consider the following the \textit{deformation space} $H^1_{\mu-1} (A)$ of the $G_2$-instanton equations at $A$:  
\begin{align*}
H^1_{\mu-1}\left(A\right) := \frac{\ker \left( * \left( * \varphi \wedge d_A  \cdot \right) :  \Omega^1_{\mu-1}\left( \mathrm{ad} P \right) \ra  \Omega^1_{\mu-2} \left( \mathrm{ad} P \right)\right)}{ \mathrm{im} \left( d_A: \Omega^0_{\mu} \left( \mathrm{ad} P \right) \ra  \Omega^1_{\mu-1} \left( \mathrm{ad} P \right)\right)}
\end{align*}
defined as the space of solutions to the linearised instanton equations $* \varphi \wedge d_A a = 0$, modulo linearised gauge transformations $d_A \Phi$ for some $\Phi \in \Omega_{\mu}^0$. 

We can also describe this space as the kernel of an elliptic operator, by fixing a choice of gauge. After this gauge-fixing, the deformation space can be identified with the kernel the Dirac operator \cite[Theorem 4.2.12]{driscollthesis} for weights $-5<\mu<0$: 
 \begin{align*}
D_A := \begin{pmatrix} 0 & d_A^* \\
d_A & * \left( * \varphi \wedge d_A \cdot\right) \\ 
\end{pmatrix} : \Omega^0_{\mu-1} \left( \mathrm{ad} P \right) \oplus  \Omega^1_{\mu-1} \left( \mathrm{ad} P \right)\ra \Omega^0_{\mu-2} \left( \mathrm{ad} P \right) \oplus  \Omega^1_{\mu-2} \left( \mathrm{ad} P \right) 
\end{align*}
Moreover, outside of some discrete set of critical weights depending only on $\left(P_\infty, A_\infty \right)$ and the geometry of asymptotic cone, $D_A$ is Fredholm, thus it has a well-defined index $\dim \mathrm{ker} D_A - \dim \mathrm{coker} D_A$. 

In suitably nice cases, one might hope that any solution of the gauge-fixed linearised equations $D_A \left(\Phi, a \right)= 0 $ can be integrated to find a solution of the full system \eqref{eq:g2instanton0}. If $D_A$ is surjective, then this holds in general by the implicit function theorem: we define $A$ to be \textit{obstructed} if this fails, i.e. if $D_A$ has a non-trivial co-kernel. 

With this general picture understood, let us return to the Bryant-Salamon metric on $\mathbf{S}(S^3)$. Any principal bundle $P \ra \mathbf{S}(S^3)$ must be trivial for gauge group $\SU(2)$, and we fix an asymptotic framing by the homogeneous bundle $P_\infty=\SU(2)^3\times_{\Delta \SU(2)} \SU(2)$ over $S^3 \times S^3$, where $\Delta \SU(2)$ acts on the gauge group via the identity map. Recall that the $\SU(2)^3$-invariant canonical connection associated to this homogeneous bundle is the nearly-K\"{a}hler instanton $A^{nK}$ considered in \S \ref{section:g2instantonodes}.

Consider the Dirac operator $D_A: \Omega^0_{\mu-1} \left( \mathrm{ad} P \right) \oplus  \Omega^1_{\mu-1} \left( \mathrm{ad} P \right) \ra \Omega^0_{\mu-2} \left( \mathrm{ad} P \right) \oplus  \Omega^1_{\mu-2} \left( \mathrm{ad} P \right)$, associated to a $G_2$-instanton $A \in \mathcal{A}_{\mu-1}$ on $\mathbf{S}(S^3)$ asymptotic to $A^{nK}$. By \cite[Thm.6.5.5]{driscollthesis}, this operator is Fredholm with index $\mathrm{ind} D_A =1$ for weights $\mu$ between $\mu \in \left( -2, -0 \right)$, and below the critical weight $\mu = -2$ the index is negative. In particular, the deformation theory is always obstructed below this critical weight, and the index matches the dimension for our invariant solutions in Theorem \ref{thm:g2inst}. 

\begin{prop}[\cite{driscollthesis}] \label{prop:driscoll} Let $P$ be the trivial bundle over the Bryant-Salamon $\mathbf{S}(S^3)$, framed at infinity by the non-trivial homogeneous bundle $P_\infty=\SU(2)^3\times_{\Delta \SU(2)} \SU(2)$ as above. Let $D_A$ be the Dirac operator associated to a $G_2$-instanton $A \in \mathcal{A}_{\mu-1}$ asymptotic to the nearly-K\"{a}hler instanton $A^{nK}$. The index of $D_A$ is $1$ for $\mu \in \left( -2, -0 \right)$, and $-1$ for $ \mu \in \left( -4, -2 \right)$.  
\end{prop} 

 \subsection{Symmetries}
We now consider the role of symmetries in this set-up. As before, let $\left(M,\varphi\right)$ be an AC $G_2$-manifold with asymptotic cone $C(\Sigma)$, and let $P \ra M$ be a principal $G$-bundle with compact Lie group $G$, framed at infinity by a $G$-bundle $P_\infty$ with a connection $A_\infty$. If $\varphi$ is invariant under a diffeomorphism $\sigma$ of $M$, we want to obtain a general criteria for understanding when the pulling back a $G_2$-instanton $A$ via some lift of $\sigma$ to the total space of $P$ is a gauge transformation of $A$. 

We can understand this at the infinitesimal level: denote by $\mathfrak{aut}(M,\varphi)$ the Lie-algebra of vector-fields on $M$ fixing the $G_2$-structure, and $\mathfrak{aut}(P_\infty,A_\infty)$ the Lie-algebra of vector fields on $P_\infty$ fixing the connection $A_\infty$. 

Suppose we have a Lie sub-algebra $\mathfrak{k} \subset \mathfrak{aut}(M,\varphi)$ of vector-fields which restrict to vector-fields pulled back from $\Sigma$ along the end, and we are given a lift $X \mapsto X_\infty$ of $\mathfrak{k}$ to $P_\infty$ for which $A_\infty$ is invariant, i.e. a Lie-algebra homomorphism $\mathfrak{k} \ra \mathfrak{aut}(P_\infty,A_\infty)$. Moreover, assume there exists an extension of this lift to the interior i.e. a lift $X \mapsto \tilde{X}_\infty$ to a vector-field on the total space of $P$, such that the vertical vector-field $\tilde{X}_\infty - X_\infty$, viewed here as a section of the adjoint bundle, lies in $\Omega^0_{\mu} \left( \mathrm{ad} P \right)$\footnote{note that such an extension always exists if $P$ admits a $K$-invariant connection asymptotic to $A_\infty$ with rate $\mu-1$.}. 

In this set-up, we prove the following lemma:  

\begin{lemma} \label{lemma:deformationmap} If $A\in \mathcal{A}_{\mu-1}$ is a $G_2$-instanton on $P$   asymptotic to $A_\infty$, then there is well-defined linear map:
\begin{align} \label{eq:deformationmap}
L:\mathfrak{k} \ra H^1_{\mu-1}\left(A\right)& &L:X \mapsto \left[\mathcal{L}_{\tilde{X}_\infty} A\right] 
\end{align} 
Moreover, $\ker{L} \subset \mathfrak{k}$ is a Lie-sub-algebra. 
\end{lemma}
\begin{proof}
 To verify that \eqref{eq:deformationmap} is well-defined, we will use the identity 
 \begin{align*}
 \mathcal{L}_{\tilde{X}} A:= d (\tilde{X} \lrcorner A) + \tilde{X} \lrcorner d A  = X \lrcorner F_A + d_A \left( \tilde{X} \lrcorner A \right)
 \end{align*}
for any lift $\tilde{X}$ to $P$ of a vector field $X$ on $M$, where we view the $G$-equivariant map $\tilde{X} \lrcorner A$ from $P$ to the Lie algebra of $G$ as a section of the adjoint bundle. We can show $\mathcal{L}_{X_\infty} A \in \Omega^1_{\mu-1} \left( \mathrm{ad} P \right)$ by restricting to the end of $M$ and setting $a = A - A_\infty$. Then for any $\Phi \in \Omega^0 \left( \mathrm{ad} P \right)$:
 \begin{align*}
 F_{A_\infty} = F_A - d_A a - \left[ a \wedge a \right]& &d_{A_\infty} \Phi = d_A \Phi - \left[ a, \Phi\right]   
 \end{align*}
Since by assumption, $\mathcal{L}_{X_\infty} {A_\infty} = X \lrcorner F_{A_\infty} + d_{A_\infty} \left( X_\infty \lrcorner A_\infty \right) = 0$, we have:  
 \begin{align*}
  \mathcal{L}_{\tilde{X}_\infty} A &= X \lrcorner F_{A_\infty} + X \lrcorner \left( d_A a + \left[ a \wedge a \right] \right) + d_A \left( \tilde{X}_\infty \lrcorner A \right) \\
  &= d_A \left( X \lrcorner a \right) + \left[ a , \left( X_\infty \lrcorner A_\infty \right) \right] + X \lrcorner \left( d_A a + \left[ a \wedge a \right] \right) + d_A \left( \tilde{X}_\infty  - X_\infty \right) 
 \end{align*}
To show that this lies in $\Omega^1_{\mu-1} \left( \mathrm{ad} P \right)$, we note that $X_\infty \lrcorner A_\infty \in \Omega^0 \left( \mathrm{ad} P_\infty \right)$ has constant norm along the end, and $X$ restricts to a vector-field pulled back from $\Sigma$, so $|X|$ grows linearly. Moreover, $a \in \Omega^1_{\mu-1} \left( \mathrm{ad} P \right)$, $\tilde{X}_\infty  - X_\infty \in \Omega^0_{\mu} \left( \mathrm{ad} P \right)$ by assumption, thus  $\mathcal{L}_{\tilde{X}_\infty} A \in \Omega^1_{\mu-1} \left( \mathrm{ad} P \right)$.   

Observe that $L(X)=0$ if and only if there is a unique lift $X \mapsto X_A$ to a vector-field on $P$ such that $\mathcal{L}_{X_A} A = 0$, and the vertical vector-field $ \Phi := X_A - \tilde{X}_\infty$ on $P$ lies in $\Omega^0_{\mu} \left( \mathrm{ad} P \right)$, viewed here as a section of the adjoint bundle. This section $\Phi$ is precisely the one that satisfies $\mathcal{L}_{\tilde{X}_\infty} A = d_A \Phi$, so uniqueness follows from the injectivity of $d_A:\Omega^0_{\mu} \left( \mathrm{ad} P \right) \ra  \Omega^1_{\mu-1} \left( \mathrm{ad} P \right)$  \cite[Cor.4.2.6]{driscollthesis}.

Now, since the lift of $\left[ X,Y \right]_\infty$ can be identified with the commutator $\left[X_\infty,Y_\infty \right]$ on $P_\infty$ for all $X,Y \in \mathfrak{k}$, then it is not hard to see that $\left[ X,Y \right]_A := \left[X_A,Y_A \right]$ also satisfies the two conditions for lifting $\left[ X,Y \right]$ to $P$ if $L(X) = L(Y) =0$, and so $\ker L \subset \mathfrak{k}$ is a Lie sub-algebra.  
\end{proof}

 We will use this observation to prove the following proposition: 

\begin{prop} \label{prop:deform} Any $G_2$-instanton on $\mathbf{S}(S^3)$ asymptotic to $A^{nK}$ with rate $-2<\mu-1<0$ is either obstructed or gauge-equivalent to an instanton in the one of the families $T_\gamma$, $T'_{\gamma'}$.
\end{prop} 

\begin{proof} 
We will the computation of Proposition \ref{prop:driscoll} to show that, if an instanton $A \in \mathcal{A}_{\mu-1}$ is not obstructed, then it must be $\SU(2)^3$-invariant, for some lift of the action of $\SU(2)^3$ to $P$ asymptotic to the action of $\SU(2)^3$ on the framing bundle $P_\infty = \SU(2)^3\times_{\Delta \SU(2)} \SU(2) \ra \SU(2)^3 / \Delta \SU(2)$. Once this is proven, the result follows from the existence and uniqueness results of Theorem \ref{thm:g2inst} in the previous section.

So to prove invariance, we note that if $A$ is not obstructed, the deformation space $H^1_{\mu-1}\left(A\right)$ is one-dimensional for $-2<\mu<0$ by \cite[Thm.6.5.5]{driscollthesis}. Since the map $L: \mathfrak{su}(2)^3 \ra H^1_{\mu-1}\left(A\right)$ defined in Lemma \ref{lemma:deformationmap} is linear, then the kernel has co-dimension at most one in $\mathfrak{su}(2)^3$. However, since this kernel is a Lie sub-algebra of $\mathfrak{su}(2)^3$, it cannot have co-dimension one, and so $L$ must vanish on all of $\mathfrak{su}(2)^3$.  

As previously discussed, this implies that we can uniquely lift $\mathfrak{su}(2)^3$ to a Lie-algebra of vector-fields on $P$ fixing $A$, such that these vector-fields are asymptotic to the infinitesimal action of $\SU(2)^3$ on the homogeneous bundle $P_\infty = \SU(2)^3\times_{\Delta \SU(2)} \SU(2) \ra \SU(2)^3/\Delta \SU(2)$. Since these vector-fields are complete, and $\SU(2)^3$ is simply-connected, it follows by \cite[Ch.3 Thm.7, Ch.4 Thm.3]{palais:globallietheory} that these vector-fields integrate to give a unique lift of the $\SU(2)^3$-action to $P$ fixing $A$.  
\end{proof}

\bibliographystyle{alpha}
\bibliography{diffgeo1} 
\end{document}